\documentclass[numbers,webpdf]{ima-authoring-template}%

\usepackage{amsmath}
\usepackage{amsfonts}
\usepackage{amssymb}
\usepackage{mathtools}
\usepackage{tikz}
\usepackage{bm}
\graphicspath{{Fig/}}

\theoremstyle{thmstyletwo}%
\newtheorem{theorem}{Theorem}
\newtheorem{proposition}[theorem]{Proposition}%

\theoremstyle{thmstyletwo}

\newtheorem{definition}{Definition}

\numberwithin{equation}{section}

\begin{document}

\DOI{DOI HERE}
\copyrightyear{2021}
\vol{00}
\pubyear{2021}
\access{Advance Access Publication Date: Day Month Year}
\appnotes{Paper}
\copyrightstatement{Published by Oxford University Press on behalf of the Institute of Mathematics and its Applications. All rights reserved.}
\firstpage{1}


\title[Backstepping Control Laws for Multi-Dimensional PDEs]{Backstepping Control Laws for Higher-Dimensional PDEs:\\ Spatial Invariance and Domain Extension Methods}

\author{Rafael Vazquez*\ORCID{0000-0001-6904-2055}
\address{\orgdiv{Departamento de Ingenieria Aeroespacial y Mecanica de Fluidos}, \orgname{Universidad de Sevilla}, \orgaddress{\street{Camino de los Descubrimiento s.n.}, \postcode{41092}, \state{Sevilla}, \country{Spain}}}}

\authormark{Rafael Vazquez}

\corresp[*]{Corresponding author: \href{email:rvazquez1@us.es}{rvazquez1@us.es}}

\received{28}{2}{2025}
\revised{Date}{0}{Year}
\accepted{Date}{0}{Year}


\abstract{
This paper extends backstepping to higher-dimensional PDEs by leveraging domain symmetries and structural properties. We systematically address three increasingly complex scenarios. First, for rectangular domains, we characterize boundary stabilization with finite-dimensional actuation by combining backstepping with Fourier analysis, deriving explicit necessary conditions. Second, for reaction-diffusion equations on sector domains, we use angular eigenfunction expansions to obtain kernel solutions in terms of modified Bessel functions. Finally, we outline a domain extension method for irregular domains, transforming the boundary control problem into an equivalent one on a target domain. This framework unifies and extends previous backstepping results, offering new tools for higher-dimensional domains where classical separation of variables is inapplicable.
}


 \interdisplaylinepenalty=2500

\maketitle

\section{Introduction}

\subsection{Higher-dimensional backstepping control of PDEs}
The control of partial differential equations has experienced remarkable developments in recent decades, with numerous methodologies emerging for designing stabilizing feedback laws. The backstepping methodology, initially developed for ordinary differential equations \cite{KKK} and later extended to PDEs \cite{krstic}, has proven particularly successful. In its original form, the method focused on one-dimensional parabolic and hyperbolic systems, providing systematic procedures for boundary feedback design through (typically) Volterra integral transformations. Subsequent advances included the use of spatial Volterra series for nonlinear systems \cite{vaz2008_a,vaz2008_b}, coupled hyperbolic 1-D systems~\cite{long-nonlinear,auriol,rijke}, coupled parabolic systems \cite{Vazquez2017,simon,leo}, adaptive control~\cite{aamo}, output regulation~\cite{deutscher}, and delays~\cite{krstic5}. See the  survey~\cite{survey} for many more applications and extensions.

The extension of backstepping to higher-dimensional domains, however, presents significant challenges due to the increased complexity of the integral transformations involved and their associated kernel equations, which become significantly harder or impossible to solve in their general form. Despite these challenges, researchers have made substantial progress by focusing on cases where special geometric properties and boundary conditions can be exploited to simplify the kernel equations and enable controller synthesis, sometimes even in explicit form. Early successes in multi-dimensional control were achieved in the context of fluid flows, based on the concept of spatial invariance \cite{bamieh2002distributed}. This property, also known as translational invariance, emerges when the system dynamics and geometry remain invariant under translations in one or more spatial coordinates. Such invariance allows for a reduction in the system's dimensionality by transforming the spatially invariant coordinates into parameters. Essentially, spatial invariance permits the substitution of spatial derivatives and spatial dependence with algebraic multiplication and dependence on new parameters that replace the spatially invariant coordinates, transforming the original high-dimension PDE system into a family (or \emph{ensemble}) of parameterized lower-dimensional systems (typically 1-D) that are easier to analyze and control.

Several works have successfully applied these concepts across different domains. The studies \cite{Vazquez2007,coc2009} addressed the control of infinite channel flows by exploiting the spatial invariance of the geometry to transform the original PDEs into ensembles of simpler one-dimensional equations parameterized by wave numbers, while \cite{vazquez2008control} examined periodic channels using similar ideas with Fourier series. This methodology has also proven effective for convection loops \cite{convloop} and magnetohydrodynamic channel flows \cite{Xu2008,vaz2008_d}. Subsequently, \cite{Vazquez2016disk} and \cite{vazquez2019sphere} developed backstepping controllers for reaction-diffusion equations defined on a 2-D disk and a 3-D sphere, respectively. By leveraging the radial symmetry inherent to these geometries, they reduced the complexity of the kernel equations and constructed explicit control laws guaranteeing exponential stability. These developments were generalized in \cite{nball} to address reaction-diffusion systems on n-dimensional balls, utilizing a combination of spherical harmonics and Bessel functions to construct the backstepping transformation and associated feedback laws. A recent advancement by \cite{R.2022ball} proposed a backstepping design for reaction-diffusion systems on balls of arbitrary dimension with spatially-dependent reaction terms, employing power series expansions to resolve the singular behavior in the kernel equations. Recent developments have also included applications to multi-agent deployment in 3-D space \cite{jie,zhang2024multi} and extensions to PDEs with boundary conditions governed by lower-dimensional PDEs \cite{jie2,vazquez-zhang}.

Alternative approaches have also emerged in the field. The work of \cite{meurer2} provided a comprehensive framework for higher-dimensional PDEs using flatness-based methods and backstepping. Domain decomposition techniques \cite{lagnese2012domain} have also proven successful in handling complex geometries, while other design methods applicable to the geometry considered in this paper include those presented in \cite{triggiani} and \cite{Barbu}.

 \subsection{Contribution and paper structure}
In this paper, we present a unified framework for extending backstepping control techniques to higher-dimensional domains by systematically exploiting their geometric properties. Our approach is based in spatial invariance principles, explained in Section~\ref{sec:spinv}, and encompasses three progressively complex scenarios:

First, in Section~\ref{sec:square}, we consider the problem of finite-dimensional control in rectangular domains, where we provide a complete characterization of the controllability conditions based on the modal decomposition of both the state and the control action. This analysis reveals fundamental limitations and possibilities in controlling higher-dimensional systems with a finite number of actuators.

Second, in Section~\ref{sec:pizza}, we address the control of reaction-diffusion equations on sector domains (which we could playfully refer to as "pizza" domains due to their shape). For these geometries, we derive explicit kernel solutions using modified Bessel functions, demonstrating how radial symmetry can be exploited to obtain tractable control laws.

Finally, in Section~\ref{sec:domain}, we tackle the challenging case of non-spatially invariant domains, exemplified by a "piano-shaped" domain that presents both theoretical and practical interests. Here, we develop a novel domain extension methodology that extends the original complex geometry into a simpler one where known control techniques can be applied, an idea suggestive of a ``target domain'' complementary of the classical backstepping target system.

We finish the paper with some conclusions and takeaways in Section~\ref{sec:concl}.

\section{Spatial Invariance and Backstepping Control in Higher Dimension}\label{sec:spinv}

A key property of many distributed parameter systems in higher dimension is their invariance with respect to translations in one or more spatial coordinates. This property, when present, allows for significant simplification of the control design problem by enabling the reduction of the system's dimensionality through spectral decomposition methods, as first systematically established by Bamieh et al. \cite{bamieh2002distributed}.
\begin{definition}[Spatial Invariance]
A system is called spatially invariant when its spatial coordinates belong to a Lie group $\mathcal{G}$, and both its dynamics and geometry are invariant with respect to the group action on these coordinates. The underlying group structure enables the transformation of the original infinite-dimensional system into a parameterized family of simpler systems through harmonic analysis.
\end{definition}
The power of this formulation lies in its generality and practical applicability. For periodic systems, $\mathcal{G}$ can be the circle group $\mathbb{S}^1$ (as in circular pipes) or the n-dimensional torus $\mathbb{T}^n$ (as in periodic channels). For unbounded domains, $\mathcal{G}$ is typically the Euclidean group $\mathbb{R}^k$ (as in infinite channels). For systems with rotational symmetry, $\mathcal{G}$ could be the special orthogonal group $SO(k)$ or the unit sphere $\mathbb{S}^k$. Each of these geometric structures naturally appears in many physical applications - from fluid flows in pipes to vehicular platoons, from cross-directional control in paper machines to networks of micro-electromechanical systems (MEMS). The key insight of Bamieh et al. \cite{bamieh2002distributed} was showing that by exploiting these symmetries through appropriate spectral transformations, the infinite-dimensional optimal control problem could be reduced to solving a parameterized family of lower dimensional problems. Perhaps even more importantly, they proved that the resulting  controllers inherit an inherent degree of spatial localization, making them particularly suitable for distributed implementation.

In the next section, we illustrate these ideas with a canonical 2D heat equation posed on a semi-infinite strip. 
We will apply a Fourier-based approach to decompose the PDE into a parameterized family of 1D problems, 
show how backstepping can be designed for each mode, and then recover the distributed control law in physical space.
Finally, we highlight how spectral truncation ensures stability with only a finite number of actively controlled modes.

\subsection{Example: Heat Equation in Semi-Infinite Strip}
As a canonical example of backstepping design through spatial invariance, consider a 2D heat equation in a semi-infinite strip $(x,y)\in(-\infty,\infty)\times[0,1]$ with Dirichlet boundary condition and a fully controlled boundary, see Fig.~\ref{fig:strip}.
\begin{align}
u_t &= \epsilon(u_{xx} + u_{yy}) + \lambda u \label{eq:heat_strip} \\
u(t,x,0) &= 0 \\
u(t,x,1) &= U(t,x) \quad \text{(control)} \label{eq:heat_strip_control}
\end{align}

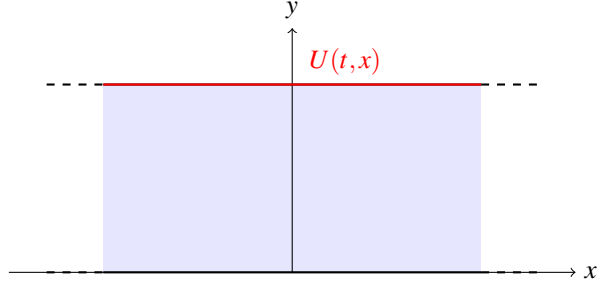
\begin{figure}
\centering
\begin{tikzpicture}[scale=2.5]
    \fill[blue!10] (-1,0) rectangle (1,1);
    
    \draw[thick] (-1,0) -- (1,0);
    \draw[thick] (-1,1) -- (1,1);
    \draw[thick,dashed] (-1.3,0) -- (-1,0);
    \draw[thick,dashed] (1,0) -- (1.3,0);
    \draw[thick,dashed] (-1.3,1) -- (-1,1);
    \draw[thick,dashed] (1,1) -- (1.3,1);
    
    \draw[->] (-1.5,0) -- (1.5,0) node[right] {$x$};
    \draw[->] (0,0) -- (0,1.3) node[above] {$y$};
    
    \draw[red,thick] (-1,1) -- (1,1) node[midway,above] {$\quad\quad\quad\quad U(t,x)$};
\end{tikzpicture}
\caption{Semi-infinite strip domain with distributed boundary control}
\label{fig:strip}
\end{figure}

This system exhibits spatial invariance with respect to the $x$ coordinate. The domain extends infinitely in the $x$ direction, with $x \in \mathbb{R}$. The control $U(t,x)$ is distributed over all $x$, while the PDE operators (specifically the Laplacian) commute with translations in $x$. Moreover, the strip geometry itself remains unchanged under translations in the $x$ direction.

For such systems, the Fourier transform provides a powerful tool to reduce the dimensionality of the control problem, as it has several important properties relevant to our control design. Parseval's identity ensures that the $L^2$ norm is preserved between physical and Fourier space, allowing us to establish stability properties in either domain. The transform converts spatial derivatives into multiplication operators, significantly simplifying the analysis of each mode. Furthermore, the natural damping of high wavenumbers provides inherent stability for the high-frequency components of the solution. On the other hand, care must be taken since real functions become complex-valued.

For functions $f(x,y)$ with $x\in(-\infty,\infty)$, we define the transform pair:
\begin{align}
\hat f(k,y) &= \int_{-\infty}^{\infty} f(x,y)e^{-2\pi i kx}dx \quad \text{(direct transform)} \\
f(x,y) &= \int_{-\infty}^{\infty} \hat f(k,y)e^{2\pi i kx}dk \quad \text{(inverse transform)}
\end{align}

For the analysis that follows, we define the spatial $L^2$ norm as
\begin{equation}
\Vert u(t,\cdot,\cdot)\Vert _{L^2(\mathbb{R}\times[0,1])}^2 = \int_0^1\int_{-\infty}^{\infty} u^2(t,x,y) dx dy
\end{equation}
and its Fourier transform counterpart
\begin{equation}
\Vert \hat{u}(t,\cdot,\cdot)\Vert _{L^2(\mathbb{R}\times[0,1])}^2 = \int_0^1\int_{-\infty}^{\infty} \vert \hat{u}(t,k,y) \vert^2 dk dy
\end{equation}
Even if $u(t,x,y)$ is real-valued, its Fourier transform $\hat{u}(t,k,y)$ 
is generally a complex-valued function of $k$. 
Hence, in the $L^2$ norm, we use $\vert\hat{f}(k,y)\vert^2 = \hat{f}(k,y)\,\overline{\hat{f}(k,y)}$ where $\overline z$ is the complex conjugate of a complex number $z$.

By Parseval's theorem \cite{rudin1962}, we have the equality
\begin{equation}
\Vert  u(t,\cdot,\cdot)\Vert _{L^2(\mathbb{R}\times[0,1])} = \Vert \hat{u}(t,\cdot,\cdot)\Vert _{L^2(\mathbb{R}\times[0,1])}
\end{equation}
Moreover, the Fourier transform constitutes an isomorphism between these $L^2$ spaces, establishing a one-to-one correspondence between functions and their transforms that will be key in analyzing both the idealized control law and its practical approximations. In what follows, we will drop the $\hat u$ notation and understand we are referring to Fourier transforms when the function depends on the wave number $k$.

Applying the Fourier transform to \eqref{eq:heat_strip}, we obtain:
\begin{align}
u_t &= \epsilon(-4\pi^2k^2u + u_{yy}) + \lambda u \\
u(t,k,0) &= 0 \\
u(t,k,1) &= U(t,k)
\end{align}
Throughout our analysis, $(x,y)\in\mathbb{R}\times[0,1]$ denote the physical-space variables, while $k\in\mathbb{R}$ denotes the wavenumber in the Fourier-transformed system. 
After applying the Fourier transform in the $x$ direction, the function $u(t,k,y)$ becomes complex-valued in $k$, 
yet the norm definitions and stability arguments remain valid by considering $|u|^2$ (rather than $u^2$) under the integral. Thus, the original 2D PDE transforms into an ensemble of (complex-valued) 1D PDEs parameterized by the wavenumber $k$. The $x$ derivatives become \emph{algebraic} terms through multiplication by $-4\pi^2k^2$, with higher wavenumbers experiencing stronger natural damping. The control design can be performed independently for each value of $k$. 

To design a backstepping controller, consider first the following target system, which is equally damped for all the wave numbers\footnote{Different wave numbers could be damped differently, e.g. to achieve an uniform decay rate independent of the wave number, but this is not sought in this paper.}
\begin{align}
w_t &= \epsilon(-4\pi^2k^2w + w_{yy}) -c w \\
w(t,k,0) &= 
w(t,k,1) = 0
\end{align}

Note that, classically as one does in regular 1D backstepping, the operator is unchanged (in this case this includes the algebraic term $-4\pi^2k^2w$ introduced by the Fourier transform). The transformation used to map the system into the target variables is defined as dependent on the wave number
\begin{equation}
w(t,k,y)=u(t,k,y)-\int_0^y K(k,y,\eta) u(t,k,\eta) d\eta
\end{equation}
which is a Volterra-type transformation and thus readily invertible with only mild requirements from the kernel $K$ (such as boundedness, see e.g.~\cite{survey}).

The resulting kernel equations are not dependent on $k$ due to the simplicity of the equations, obtaining the classical kernel equations for a constant-coefficient 1D reaction-diffusion equation:
\begin{align}
K_{yy}(k,y,\eta)-K_{\eta\eta}(k,y,\eta) &= \lambda_0 K(k,y,\eta) \\
K (k,y,0) &= 0 \\
K(k,y,y)& = -\frac{\lambda_0 y}{2}
\end{align}
where $\lambda_0=\frac{\lambda+c}{\epsilon}$,  obtaining~\cite{krstic} $K(k,y,\eta)=\lambda_0 \eta \frac{\mathrm{I}_1(\sqrt{\lambda_0(y^2-\eta^2)})}{\sqrt{\lambda_0(y^2-\eta^2)}}$, with $\mathrm{I}_1$ the first-order modified Bessel function of the first kind.

Then by considering as usual the boundary conditions of original and target systems at $x=1$, and the backstepping transformation, one can deduce that the control law in Fourier space takes the form:
\begin{equation}
U(t,k) = -\int_0^1 \lambda_0 \eta \frac{I_1(\sqrt{\lambda_0(1-\eta^2)})}{\sqrt{\lambda_0(1-\eta^2)}} u(t,k,\eta) d\eta
\end{equation}
where $\lambda_0 = \frac{\lambda+c}{\epsilon}$. The physical control is then recovered via the inverse transform:
\begin{equation}
U(t,x) = \int_{-\infty}^{\infty} U(t,k)e^{2\pi i kx}dk
\end{equation}

The physical space control law can be written more explicitly by substituting the Fourier transform of $u(t,x,y)$:
\begin{align}
U(t,x) &= -\int_0^1 \int_{-\infty}^{\infty} \lambda_0 \eta \frac{I_1(\sqrt{\lambda_0(1-\eta^2)})}{\sqrt{\lambda_0(1-\eta^2)}}  \left(\int_{-\infty}^{\infty} u(t,\xi,\eta)e^{-2\pi i k\xi}d\xi\right) e^{2\pi i kx} d\eta dk
\end{align}

After exchanging the order of integration and using the Fourier transform identity $\int_{-\infty}^{\infty} e^{2\pi i k(x-\xi)}dk = \delta(x-\xi)$, we obtain:
\begin{align}
U(t,x) &= -\int_0^1 \int_{-\infty}^{\infty} \lambda_0 \eta \frac{I_1(\sqrt{\lambda_0(1-\eta^2)})}{\sqrt{\lambda_0(1-\eta^2)}} \delta(x-\xi) u(t,\xi,\eta) d\eta \nonumber \\
&=
-\int_0^1 \lambda_0 \eta \frac{I_1(\sqrt{\lambda_0(1-\eta^2)})}{\sqrt{\lambda_0(1-\eta^2)}} u(t,x,\eta) d\eta
\end{align}
The appearance of the Dirac delta $\delta(x-\xi)$ in the inverse transform is a manifestation of 
the Fourier principle that a sum (or integral) over all wavenumbers in $\mathbb{R}$ 
reconstructs functions perfectly in physical space. 
Here, controlling each $k$-mode independently ``sums up'' to yield a purely local kernel 
in the $x$ direction, reflecting a remarkable level of spatial locality in the resulting control law. And we have obtained an explicit formula for the gain kernel of our control law, even if it is distributional, namely $K_1(x,\xi,\eta)=-\lambda_0 \eta \frac{I_1(\sqrt{\lambda_0(1-\eta^2)})}{\sqrt{\lambda_0(1-\eta^2)}} \delta(x-\xi)$.

Considering on the other hand that higher wave numbers are naturally stable, we can alternatively introduce a spectral truncation to wavenumbers $|k| \leq N$:
\begin{equation}
U_N(t,x) = -\int_0^1 \int_{-\infty}^{\infty} h(\eta) \left(\int_{-N}^N u(t,\xi,\eta)e^{2\pi i k(x-\xi)}dk\right) d\xi d\eta
\end{equation}
where
\begin{equation}
h(\eta) = \lambda_0 \eta \frac{\mathrm{I}_1(\sqrt{\lambda_0(1-\eta^2)})}{\sqrt{\lambda_0(1-\eta^2)}}
\end{equation}

Computing the inner integral yields:
\begin{align}
U_N(t,x) &= -\int_0^1 \int_{-\infty}^{\infty} h(\eta) [2N\text{sinc}(2\pi N(x-\xi))] u(t,\xi,\eta)d\xi d\eta \nonumber \\
&= -\int_0^1 \int_{-\infty}^{\infty} K_{1,N} (x,\xi,\eta)u(t,\xi,\eta)d\xi d\eta \label{eqn-Un}
\end{align}
where
\begin{equation}
K_{1,N}(x,\xi,\eta) = 2N h(\eta)\text{sinc}(2\pi N(x-\xi))
\end{equation}
and where $\text{sinc}(z) = \frac{\sin(z)}{z}$ is the cardinal sine function. The kernel thus splits naturally into two parts: the Bessel function term arising from the backstepping design, and the sinc function term from the wavenumber cutoff.
This truncated kernel $K_{1,N}$ converges to $h(\eta)\delta(x-\xi)$ as $N \to \infty$ in the sense of distributions, providing a formal link between the full spectrum and truncated implementations. 

The parameter $N$ thus serves a dual role: It determines the spatial locality of the control law in $x$, with larger $N$ giving closer approximation to the ideal delta function kernel (highly-localized control law), and it sets the number of actively controlled modes needed to achieve a desired stability margin, as formalized next.

\begin{theorem}[Stability with Spectral Truncation]  \label{th-stab} 
For any given decay rate $c > 0$, set $
N_0 = \sqrt{\frac{c+\lambda}{4\pi^2\epsilon}}$ and let $N\in \mathbb{N}$ such that $N \geq N_0$. For the closed-loop system (\ref{eq:heat_strip})--(\ref{eq:heat_strip_control}) with the
truncated feedback control law $U_N$ in (\ref{eqn-Un}),  the equilibrium $u(t,x,y)\equiv 0$ achieves exponential stability with decay rate $c$:
\begin{equation}
\Vert u(t,\cdot,\cdot)\Vert _{L^2(\mathbb{R}\times[0,1])} \leq M e^{-ct}\Vert u(0,\cdot,\cdot)\Vert _{L^2(\mathbb{R}\times[0,1])}
\end{equation}
for some $M\geq1$ that does not depend on $N$.

\end{theorem}
\begin{proof}
The proof proceeds in two steps, analyzing separately the controlled and uncontrolled wavenumbers, to prove the desired exponential rate of convergence, and finally stitches both results together with the help of Parseval's Theorem.

Step 1 (Controlled wavenumbers, $|k| < N$):
Consider the Lyapunov function for the target system
\begin{equation}
V_k(t) = \frac{1}{2} \Vert w(t,k,\cdot)\Vert_{L^2[0,1]}^2 = \frac{1}{2} \int_0^1 \vert w(t,k,\eta) \vert ^2 d\eta
\end{equation}
Taking its time derivative along solutions of the target system:
\begin{align}
\dot{V}_k(t) &= \text{Re}\left\{\int_0^1 w(t,k,\eta)\overline{w_t(t,k,\eta)}d\eta \right\} \nonumber \\
&= \text{Re}\left\{ \int_0^1 w(t,k,\eta)\overline{[\epsilon(\partial\eta^2 - 4\pi^2k^2) - c]w(t,k,\eta)}d\eta \right\} \nonumber \\
&= -2c V_k + \epsilon\text{Re}\left\{\int_0^1 w(t,k,\eta)\overline{ w_{\eta \eta} (t,k,\eta)}d\eta \right\}- 8\pi^2\epsilon k^2 V_k
\end{align}
where $\text{Re}\left\{ \cdot\right\}$ denotes real part. Using integration by parts and the vanishing boundary conditions of $w$, the middle term becomes:
\begin{equation}
\epsilon\text{Re}\left\{\int_0^1 w(t,k,\eta)\overline{ w_{\eta \eta}(t,k,\eta)}d\eta\right\} = -\epsilon \Vert w_{\eta}(t,k,\cdot)\Vert_{L^2[0,1]}^2 \leq 0
\end{equation}
Therefore:
\begin{equation}
\dot{V}_k(t) \leq  -2cV_k(t)
\end{equation}
and we readily obtain
$\Vert w(t,k,\cdot)\Vert_{L^2[0,1]} \leq \mathrm{e}^{-ct} \Vert w(0,k,\cdot)\Vert_{L^2[0,1]}$.
Now, using the direct and inverse transformations
\begin{align}
w(t,k,\eta) &= u(t,k,\eta) - \int_0^\eta K(k,\eta,y)u(t,k,y)dy \\
u(t,k,\eta) &= w(t,k,\eta) + \int_0^\eta L(k,\eta,y)w(t,k,y)dy
\end{align}
and the $L^\infty$ bounds on the kernels, see e.g.~\cite{krstic}, which are independent of $k$ in this case,
we can establish
\begin{equation}
\Vert  u(t,k,\eta) \Vert_{L^2[0,1]}^2 \leq (1 +  \Vert L(k,\cdot,\cdot)\Vert _{L^\infty})^2 \Vert  w(t,k,\cdot)\Vert_{L^2[0,1]}^2
\end{equation}
and
\begin{equation}
\Vert  w(0,k,\eta) \Vert_{L^2[0,1]}^2 \leq (1 + \Vert K(k,\cdot,\cdot)\Vert_{L^\infty})^2 \Vert  y(0,k,\cdot)\Vert_{L^2[0,1]}^2
\end{equation}
which implies
\begin{equation}
\Vert u(t,k,\cdot)\Vert_{L^2[0,1]} \leq M \mathrm{e}^{-ct} \Vert u(0,k,\cdot)\Vert_{L^2[0,1]}
\end{equation}
for $M=(1 +  \Vert L(k,\cdot,\cdot)\Vert _{L^\infty}) (1 + \Vert K(k,\cdot,\cdot)\Vert_{L^\infty})\geq 1$.

Step 2 (Uncontrolled wavenumbers, $|k| \geq N$):
For these modes, consider just the Lyapunov function
\begin{equation}
W_k(t) = \frac{1}{2} \Vert u(t,k,\cdot)\Vert_{L^2[0,1]}^2 = \int_0^1 \vert u(t,k,\eta) \vert ^2 d\eta
\end{equation}
Taking its time derivative along solutions of the uncontrolled PDE:
\begin{align}
\dot{W}_k(t) &= \text{Re}\left\{\int_0^1 u(t,k,\eta) \overline{[\epsilon(\partial\eta^2 - 4\pi^2k^2) + \lambda]u(t,k,\eta)}d\eta\right\} \nonumber \\
&\leq 2(\lambda-4\pi^2\epsilon k^2)W_k
\end{align}
by integrating by parts and applying the boundary conditions as in Step 1.

Therefore, choosing
\begin{equation}
N > \sqrt{\frac{c+\lambda}{4\pi^2\epsilon}}
\end{equation}
ensures $\dot{W}_k(t) \leq -2cW_k(t)$ for all $|k| \geq N$. Thus we directly get, for $\vert k\vert>N$ that $\Vert u(t,k,\cdot)\Vert_{L^2[0,1]} \leq \mathrm{e}^{-ct} \Vert u(0,k,\cdot)\Vert_{L^2[0,1]}$

The global result follows by combining both estimates to cover all wave numbers and using Parseval's theorem twice:
\begin{align}
\Vert u(t,\cdot,\cdot) \Vert_{L^2(\mathbb{R}\times[0,1])}^2 &= \int_{-\infty}^{\infty} \Vert u(t,k,\cdot) \Vert_{L^2[0,1]}^2 dk
\nonumber \\
&\leq M \mathrm{e}^{-ct}  \int_{-\infty}^{\infty} \Vert u(0,k,\cdot) \Vert_{L^2[0,1]}^2 dk \nonumber \\
&=M \mathrm{e}^{-ct} \Vert u(0,\cdot,\cdot) \Vert_{L^2(\mathbb{R}\times[0,1])}^2
\end{align}
\end{proof}

This example illustrates a general principle that applies to reaction-diffusion systems in higher dimensions: there exists a finite number $N$ of modes requiring active control, while higher modes remain naturally stable through diffusive damping. This number $N$ depends explicitly on physical parameters ($\epsilon$, $\lambda$) and the desired decay rate $c$.

The control kernel's structure itself reflects this duality of mechanisms:
\begin{equation}
K(x,\xi,\eta) = h(\eta)[2N\text{sinc}(2\pi N(x-\xi))]
\end{equation}
where $h(\eta)$ arises from the backstepping design (containing the modified Bessel function), while the sinc term emerges from spectral truncation. This decomposition reveals how spatial locality interacts with spectral properties of the control law.
These insights prove fundamental when addressing more complex geometries. While the technical details become more involved, the core principles - modal decomposition, natural damping at high frequencies, and the separation into actively controlled and naturally stable modes - persist in other domains. This framework will guide our subsequent analysis of both rectangular domains with finite-dimensional actuation and in sector domains (where radial symmetry replaces translational invariance).

\section{Backstepping Control of Reaction-Diffusion Systems in Square Domains}\label{sec:square}

The square domain presents an interesting intermediate case between the fully spatially invariant systems discussed previously and domains with no spatial symmetries. While translation invariance is apparently lost, the regular geometry still enables powerful decomposition methods that connect to our earlier analysis. We begin by examining a reaction-diffusion system on the unit square, as shown in Fig.~\ref{fig:square}:

\begin{figure}
\centering
\begin{tikzpicture}[scale=2.5]
    \fill[blue!10] (0,0) rectangle (1,1);
    
    \draw[thick] (0,0) rectangle (1,1);
    
    \draw[->] (-0.2,0) -- (1.3,0) node[right] {$x$};
    \draw[->] (0,-0.2) -- (0,1.3) node[above] {$y$};
    
    \draw[red,thick] (1,0) -- (1,1) node[midway,right] {$U(t,y)$};
    
    \draw[blue,thick] (0,0) -- (1,0) node[midway,below] {$u=0$};
    \draw[blue,thick] (0,1) -- (1,1) node[midway,above] {$u=0$};
    \draw[blue,thick] (0,0) -- (0,1) node[midway,left] {$u=0$};
\end{tikzpicture}
\caption{Square domain with boundary control on the right edge}
\label{fig:square}
\end{figure}
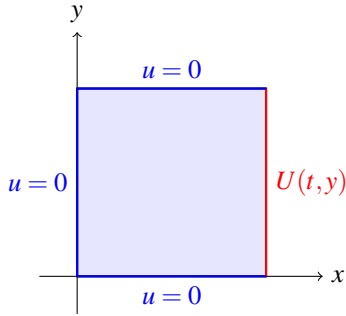

\begin{align}
u_t &= \epsilon(u_{xx} + u_{yy}) + \lambda u, \quad (x,y) \in [0,1]\times[0,1] \label{eqn:usquare}\\
u(t,0,y) &= 0, \quad u(t,1,y) = U(t,y) \quad \text{(control)} \\
u(t,x,0) &= u(t,x,1) = 0 \quad \text{(Dirichlet boundary conditions)} \label{eqn:usquarebc}
\end{align}
For $\epsilon,\lambda>0$. If $\lambda$ is sufficiently large, the system is unstable.

Interestingly, this problem is solvable without much complication due to the simplicity of the geometry and boundary conditions. Indeed, if one follows the backstepping approach, we seek a transformation:
\begin{equation}
w(t,x,y) = u(t,x,y) - \int_0^x k(x,\xi)u(t,\xi,y)d\xi \label{eqn:basic_transf}
\end{equation}
The kernel $k(x,\xi)$ must map our system to a target system with enhanced stability:
\begin{align}
w_t &= \epsilon(w_{xx} + w_{yy}) - cw \\
w(t,0,y) &= w(t,1,y) = w(t,x,0) = w(t,x,1) = 0
\end{align}
Now, deriving the kernel equations, a remarkable feature of this transformation is that it commutes with y-derivatives, allowing the kernel equations to be posed and solved independently of the y coordinate, reaching again the kernel equations:
\begin{align}
K_{xx}(x,\xi)-K_{\eta\eta}(x,\xi) &= \lambda_0 K(x,\xi) \\
K (x,0) &= 0 \\
K(x,x)& = -\frac{\lambda_0 x}{2}
\end{align}
where, as in  Section~\ref{sec:spinv}, $\lambda_0=\frac{\lambda+c}{\epsilon}$,  obtaining again $K(x,\xi)=\lambda_0 \xi \frac{\mathrm{I}_1(\sqrt{\lambda_0(x^2-\xi^2)})}{\sqrt{\lambda_0(x^2-\xi^2)}}$.

Thus as in the semi-infinite strip, we obtain the \emph{localized} control law
\begin{equation}
U(t,y) = \int_0^1 K(1,\xi)u(t,\xi,y)d\xi \label{eq:allmodecontrol}
\end{equation}

We skip the stability result, which is obvious. Instead, we show how to recover this result with modal methods. While this domain lacks the infinite extent that enabled Fourier transform methods in our previous analysis, its regular geometry suggests a natural spectral approach through Fourier series. Specifically, the homogeneous Dirichlet conditions in y motivate expanding the solution in sine series:
\begin{equation}
u(t,x,y) = \sum_{n=1}^{\infty} u_n(t,x)\sin(n\pi y),\quad U(t,y) = \sum_{n=1}^{\infty} U_n(t)\sin(n\pi y)
\end{equation}
This expansion effectively decomposes the 2D problem into countably many 1D problems (an ensemble), each corresponding to a different sine mode. For each mode n:
\begin{align}
u_{n,t} &= \epsilon(u_{xx} - n^2\pi^2 u_n) + \lambda u_n \\
u_n(t,0) &= 0, \quad u_n(t,1) = U_n(t)
\end{align}

The connection to our previous spatial invariance analysis in Section~\ref{sec:spinv} becomes clear: each mode behaves like a 1D reaction-diffusion equation with an additional damping term $-\epsilon n^2\pi^2$ from the y-derivatives and can be controlled independently. This additional damping increases quadratically with the mode number, suggesting that higher modes will be naturally more stable. Despite this property, to recover the control law (\ref{eq:allmodecontrol}) we need to equally damp all modes, and thus we design a target system
\begin{align}
w_{n,t} &= \epsilon(w_{n,xx} - n^2\pi^2 w_n) - cw_n \\
w_n(t,0) &= w_n(t,1) = 0
\end{align}
The backstepping transformation in modal form becomes:
\begin{equation}
w_n(t,x) = u_n(t,x) - \int_0^x K_n(x,\xi)u_n(t,\xi)d\xi
\end{equation}
and we obtain the same kernel equations as before, independent of the mode. Thus $K_n(x,\xi)=K(x,\xi)$ and it has again the familiar form involving modified Bessel functions.

The control input for each mode is then:
\begin{equation}
U_n(t) = \int_0^1 K(1,\xi)u_n(t,\xi)d\xi
\end{equation}
To recover the physical space control law, we substitute these modal controls back into the sine series:
\begin{align}
U(t,y) &= \sum_{n=1}^{\infty} \left(\int_0^1 K(1,\xi)u_n(t,\xi)d\xi\right)\sin(n\pi y)\nonumber   \\
&= \sum_{n=1}^{\infty} \left(\int_0^1 K(1,\xi)\left[2\int_0^1 u(t,\xi,\eta)\sin(n\pi \eta)d\eta\right]d\xi\right)\sin(n\pi y) \nonumber\\
&= \int_0^1 K(1,\xi) )\int_0^1 u(t,\xi,\eta) \left[2\sum_{n=1}^{\infty} \sin(n\pi y) \sin(n\pi \eta)\right]d\eta d\xi\nonumber \\
&= \int_0^1\int_0^1 K(1,\xi)u(t,\xi,\eta)\delta(y-\eta) d\eta d\xi \nonumber \\
&= \int_0^1 K(1,\xi)u(t,\xi,y)d\xi \label{control-law-square}
\end{align}
where we've used the fact that $2\sum_{n=1}^{\infty} \sin(n\pi y)\sin(n\pi \eta) = \delta(y-\eta)$ for $y,\eta \in [0,1]$.
This explicitly shows how the control law remains local in the y coordinate - the Dirac delta ensures that the control at any y only depends on the state at that same y-coordinate.
This remarkable result shows that the control law remains local in the $y$ coordinate, matching exactly what we found through ``direct'' backstepping. The interchange of summation and integration is justified by the regularity of the kernel and the convergence properties of the Fourier series.

The use of Fourier series enables a result that is not possible with the basic transformation (\ref{eqn:basic_transf}), as explained next.

\subsection{Extension to Finite-Dimensional Control}
In practical implementations, it is rarely possible to apply arbitrary control functions $U(t,y)$ along the boundary. Instead, control actuation is typically constrained to be finite-dimensional:
\begin{equation}
U(t,y) = \sum_{k=1}^m U_k(t)\phi_k(y)
\end{equation}
where ${\phi_k(y)}_{k=1}^m$ are prescribed shape functions on $[0,1]$ and $U_k(t)$ are the $m$ control inputs. In this case, it is no longer possible to apply the ``direct'' backstepping transformation  (\ref{eqn:basic_transf})  whereas spatial invariance methods are applicable and allow to find conditions under which the problem is solvable.

Indeed, these shape functions can be expanded in the same sine basis:
\begin{equation}
\phi_k(y) = \sum_{n=1}^{\infty} \phi_{k,n}\sin(n\pi y)
\end{equation}
where $\phi_{k,n}$ are the Fourier coefficients of the shape functions.
For each mode $n$, the boundary condition at $x=1$ becomes:
\begin{equation}
u_n(t,1) = \sum_{k=1}^m U_k(t)\phi_{k,n}
\end{equation}
However, our spectral analysis reveals that not all modes require active control. Indeed, examining the uncontrolled dynamics of each mode:
\begin{equation}
u_{n,t} = \epsilon(u_{xx} - n^2\pi^2 u_n) + \lambda u_n
\end{equation}
we observe that the term $-\epsilon n^2\pi^2$ provides increasingly strong natural damping for higher modes. Specifically, for any desired decay rate $c > 0$, modes with $n > \sqrt{\frac{c+\lambda}{\pi^2\epsilon}}$ will decay faster than $e^{-ct}$ even without control.
This observation leads to a fundamental result about finite-dimensional control: to achieve a decay rate $c$, we only need enough actuators to control modes up to some $N> N_0 = \sqrt{\frac{c+\lambda}{\pi^2\epsilon}}$. The control law takes the form:
\begin{equation}
\begin{pmatrix} U_1(t) \\ \vdots \\ U_m(t) \end{pmatrix} =
\Phi^\dagger \begin{pmatrix} g_1(t) \\ \vdots \\ g_N(t) \end{pmatrix}
\end{equation}
where $\Phi$ is the matrix of shape function coefficients:
\begin{equation}
\Phi = \begin{pmatrix}
\phi_{1,1} & \cdots & \phi_{m,1} \\
\vdots & \ddots & \vdots \\
\phi_{1,N} & \cdots & \phi_{m,N}
\end{pmatrix}
\end{equation}
and
\begin{equation}
g_n(t) = \int_0^1 K(1,\xi)u_n(t,\xi)d\xi
\end{equation}
In addition, $\Phi^\dagger$ referes to the Moore-Penrose pseudoinverse. It exists and provides a meaningful control law precisely when $\Phi$ has full row rank $N$. This means the shape functions must be able to independently actuate each of the first $N$ modes. Mathematically, this requires $m \geq N$ actuators (a necessary but not sufficient condition), and the rows of $\Phi$ must be linearly independent. When these conditions are met, $\Phi^\dagger = \Phi^\top(\Phi\Phi^\top)^{-1}$, providing the exact minimum-norm solution to the modal control problem.
Note thought that the condition number of $\Phi$ may play an important role in implementation. 

If shape functions can be chosen, their choide presents an important practical tradeoff. While sinusoidal shape functions $\phi_k(y) = \sin(k\pi y)$ are mathematically optimal, leading to $\Phi$ being a submatrix of the identity, they may be challenging to implement physically. A more practical alternative is to use localized actuators:
\begin{equation}
\phi_k(y) = \begin{cases}
1 & \text{if } \frac{k-1}{m} \leq y \leq \frac{k}{m} \\
0 & \text{otherwise}
\end{cases}
\end{equation}
These piecewise constant functions correspond to independent actuators placed along the boundary. Their Fourier coefficients can be computed explicitly:
\begin{equation}
\phi_{k,n} = \frac{2}{n\pi}\left[\cos(n\pi\frac{k-1}{m}) - \cos(n\pi\frac{k}{m})\right]
\end{equation}
For these piecewise constant actuators, the condition number grows approximately linearly with $N$, reflecting the inherent challenge of controlling higher modes with localized actuators. This growth in condition number manifests physically as increasing sensitivity to measurement noise and modeling uncertainties for higher modes.

The complete physical space implementation requires computing the Fourier coefficients of the state through:
\begin{equation}
u_n(t,x) = 2\int_0^1 u(t,x,\eta)\sin(n\pi \eta)d\eta
\end{equation}
The finite-dimensional control law then becomes:
\begin{align}
U(t,y) &= \sum_{k=1}^m \phi_k(y)\left[\Phi^\dagger \begin{pmatrix}
\int_0^1 K(1,\xi)\left(2\int_0^1 u(t,\xi,\eta)\sin(\pi \eta)d\eta\right)d\xi \\
\vdots \\
\int_0^1 K(1,\xi)\left(2\int_0^1 u(t,\xi,\eta)\sin(N\pi \eta)d\eta\right)d\xi
\end{pmatrix}\right]_k \label{eqn:findimlaw}
\end{align}
This reveals an elegant interplay between spatial and spectral properties: while each actuator has localized spatial influence through $\phi_k(y)$, its effect is distributed across modes through $\Phi^\dagger$. The minimum number of actuators needed scales with the square root of the desired decay rate, a fundamental limitation inherent to the parabolic nature of the system.

The main result is given next; the use of truncation necessitates the use of the $H^1$ norm in this case.

\begin{theorem}[Stability with Finite-Dimensional Control]
Given a desired decay rate $c > 0$, let $N\in \mathbb{N}$ be an integer such that $N\geq N_0 = \sqrt{\frac{c+\lambda}{\pi^2\epsilon}}$ and consider the closed-loop system (\ref{eqn:usquare})--(\ref{eqn:usquarebc}) with $m$ actuators as given by (\ref{eqn:findimlaw}). Assume as well the shape function $\phi_n(x)$ to be in $L^2([0,1])$. If $\text{rank}(\Phi) = N$, then the equilibrium $u(t,x,y)\equiv 0$ achieves exponential stability with decay rate $c$ in the $H^1$ norm.
\begin{equation}
\Vert u(t,\cdot,\cdot)\Vert_{H^1([0,1]^2)} \leq M e^{-ct}\Vert u(0,\cdot,\cdot)\Vert_{H^1([0,1]^2)}
\end{equation}
for $M>0$.
\end{theorem}
\begin{proof}
The proof proceeds as the one of Theorem~\ref{th-stab} by analyzing separately the controlled and uncontrolled modes.

Define
\begin{eqnarray}
V_{1,n} &=& \frac{1}{2}\int_0^1 z_n^2(x,t)dx \\
V_{2,n} &=& \frac{1}{2}\int_0^1 z_{n,x}^2(x,t)dx
\end{eqnarray}
where $z_n=w_n$ if $n\leq N$, and $z_n=u_n$ if $n>N$.

For the controlled modes ($n\leq N$), we have
\begin{align}
w_{nt} &= \epsilon(-n^2\pi^2 w_n + w_{nyy}) -c w_n \\
w_n(t,0) &= 
w_n(t,1) = 0
\end{align}

It is straightforward to obtain, from the boundary conditions and integration by parts, that
\begin{align}
\dot{V}_{1,n} &= -\epsilon \int_0^1 w_{nx}^2 \, dx - \epsilon n^2\pi^2 \int_0^1 w_n^2 \, dx - c \int_0^1 w_n^2 \, dx \nonumber\\
&= -2\epsilon V_{2,n} - 2 (\epsilon n^2\pi^2 + c) V_{1,n}
\end{align}
and
\begin{align}
\dot{V}_{2,n} &= -\epsilon \int_0^1 w_{nxx}^2 \, dx - \epsilon n^2\pi^2 \int_0^1 w_{nx}^2 \, dx - c \int_0^1 w_{nx}^2 \, dx \nonumber\\
&= -\epsilon \int_0^1 w_{nxx}^2 \, dx - 2 (\epsilon n^2\pi^2 + c) V_{2,n}
\end{align}

For the uncontrolled modes, we have
\begin{align}
u_{nt} &= \epsilon(-n^2\pi^2 u_n + u_{nyy}) + \lambda u_n \\
u_n(t,0) &= 0 \\
u_n(t,1) &= U_n(t)\end{align}
with
\begin{equation}
 U_n(t)= \sum_{k=1}^m \phi_{k,n} \left[\Phi^\dagger \begin{pmatrix}
\int_0^1 K(1,\xi)u_1(t,\xi) d\xi \\
\vdots \\
\int_0^1 K(1,\xi)u_N(t,\xi) d\xi
\end{pmatrix}\right]_k=\sum_{k=1}^m \phi_{k,n} \left[\Phi^\dagger \begin{pmatrix}
\int_0^1 L(1,\xi)w_1(t,\xi) d\xi \\
\vdots \\
\int_0^1 L(1,\xi)w_N(t,\xi) d\xi
\end{pmatrix}\right]_k
\end{equation}
where we have used the inverse transformation to express the boundary feedback law in terms of the target variables.

Note the following identities:
\begin{eqnarray}
 \vert U_n(t) \vert^2  \leq  Nm \Vert L \Vert_{L^\infty}^2 \Vert \Phi^\dagger  \Vert^2   \sum_{j=1}^N \sum_{k=1}^m \vert \phi_{k,n}  \vert^2  \int_0^1  w_j^2(t,\xi) d\xi 
 \leq C_{n} \sum_{j=1}^N V_{1,j}
\end{eqnarray}
and
\begin{eqnarray}
 \vert U_{nt}(t) \vert^2  &\leq & Nm \Vert L \Vert_{L^\infty}^2 \Vert \Phi^\dagger  \Vert^2   \sum_{n=1}^N \sum_{k=1}^m \vert \phi_{k,n}  \vert^2  \int_0^1  w_{nt}^2(t,\xi) d\xi  \nonumber \\
 &\leq& C_{n}  \sum_{j=1}^N  
 \left[ 
 \epsilon^2  \int_0^1  w_{jxx}^2(t,\xi) d\xi   +  (\epsilon n^2\pi^2 + c)^2 V_{1,j}
 \right]
\end{eqnarray}
for a finite constant $C_n$.

Note as well that, by the Parseval identity, if the shape functions are well-behaved, e.g. in $L^2$, we obtain
\begin{equation}
\sum_{n=1}^\infty C_n\leq C
\end{equation}

Computing now the derivative of the Lyapunov functions for the uncontrolled modes, we obtain
\begin{align}
\dot{V}_{1,n} &= -\epsilon \int_0^1 u_{nx}^2 \, dx - \epsilon n^2\pi^2 \int_0^1 u_n^2 \, dx +\lambda  \int_0^1 u_n^2 \, dx 
+\epsilon u_{nx}(1,t) U_n(t)
\nonumber\\
&= -2\epsilon V_{2,n} - 2 (\epsilon n^2\pi^2 - \lambda) V_{1,n}+\epsilon u_{nx}(1,t) U_n(t)
\end{align}
and
\begin{align}
\dot{V}_{2,n} &= -\epsilon \int_0^1 u_{nxx}^2 \, dx - \epsilon n^2\pi^2 \int_0^1 u_{nx}^2 \, dx +\lambda \int_0^1 u_{nx}^2 \, dx+\epsilon u_{nx}(1,t) U_{nt}(t) \nonumber\\
&= -\epsilon \int_0^1 u_{nxx}^2 \, dx - 2 (\epsilon n^2\pi^2 -\lambda) V_{2,n} +\epsilon u_{nx}(1,t) U_{nx}(t)
\end{align}

Now we can use the fact that
 \begin{equation}
 f(1)=\int_0^1 \left[\frac{d}{dx} xf(x)\right] dx =\int_0^1 f(x) dx + \int_0^1 xf'(x) dx
 \end{equation}
 which squared and applying several inequalities gives
 \begin{equation}
 f^2(1) \leq 2 \left( \int_0^1 f^2(x) dx + \int_0^1 f_x^2(x) dx \right)
 \end{equation}
 Thus we get the classical bound for a trace:
 \begin{equation} u^2_x(1,t) \leq 2 \left(\int_0^1 u_{xx}^2 (x,t) dx+\int_0^1 u_{x}^2 (x,t) dx\right)
 \end{equation}
 
 Then, the Lyapunov functions for the uncontrolled modes can be bounded as follows:
 \begin{align}
\dot{V}_{1,n} &\leq -2\epsilon V_{2,n} - 2 (\epsilon n^2\pi^2 - \lambda) V_{1,n}+
\delta_1 \epsilon   \left(\int_0^1 u_{nxx}^2 (x,t) dx+V_{2n} \right)
+\frac{\epsilon}{2 \delta_1} \vert  U_n(t) \vert ^2
\end{align}
and
\begin{align}
\dot{V}_{2,n} &\leq-\epsilon \int_0^1 u_{nxx}^2 \, dx - 2 (\epsilon n^2\pi^2 -\lambda) V_{2,n} +
\delta_2 \epsilon   \left(\int_0^1 u_{nxx}^2 (x,t) dx+V_{2n} \right)
+\frac{\epsilon}{2 \delta_2} \vert  U_{nt}(t) \vert ^2
\end{align}

Consider now $V_n= V_1 + V_2$.

For the controlled modes $n\leq N$ we obtain
\begin{align}
\dot V_n &=-2 (\epsilon (n^2\pi^2+1) + c) V_{2,n} - 2 (\epsilon n^2\pi^2 + c) V_{1,n} -\epsilon \int_0^1 w_{nxx}^2 \, dx \nonumber \\
&\leq -2c V_n- 2 \epsilon \pi^2  V_{1,n}
 -\epsilon \int_0^1 w_{nxx}^2 \, dx
\end{align}
For the uncontrolled modes $n>N$ we obtain  
 \begin{align}
\dot{V}_{n} \leq &
-(2 \epsilon+ 2 (\epsilon n^2\pi^2 -\lambda) -\delta_2 \epsilon-\delta_1 \epsilon) V_{2,n} - 2  (\epsilon n^2\pi^2 - \lambda) V_{1,n} \nonumber \\
&
-\epsilon (1-\delta_1-\delta_2) \int_0^1 u_{nxx}^2 \, dx 
+\frac{ \epsilon}{2 \delta_1} \vert  U_n(t) \vert ^2+\frac{\epsilon}{2 \delta_2} \vert  U_{nt}(t) \vert ^2
\end{align}

Note that since for all uncontrolled modes  $n>N\geq N_0 = \sqrt{\frac{c+\lambda}{\pi^2\epsilon}}$, we get $n^2 > \frac{c+\lambda}{\pi^2\epsilon}$ and therefore $\epsilon n^2\pi^2 -\lambda> c$. Then, choosing $\delta_1=\delta_2=\frac{1}{4}$
 \begin{align}
\dot{V}_{n} \leq &
-\left(\frac{3}{2}  \epsilon+ 2 c \right) V_{2,n} - 2  c V_{1,n} 
-\frac{\epsilon}{2}  \int_0^1 u_{nxx}^2 \, dx 
+2 \epsilon \vert  U_n(t) \vert ^2+2 \epsilon \vert  U_{nt}(t) \vert ^2
\nonumber \\
\leq & -2c V_n +2 \epsilon ( \vert  U_n(t) \vert ^2+ \vert  U_{nt}(t) \vert ^2)
\end{align}

Consider now 
\begin{equation}
V(t)=\alpha \sum_{n=1}^N V_n(t) + \sum_{n=N+1}^\infty V_n(t) 
\end{equation}
We obtain
\begin{equation}
\dot V \leq -2c V  - 2
\left( 
\alpha \epsilon \pi^2
-C \epsilon (1+\epsilon N^2\pi^2 + c)^2
\right)\sum_{n=1}^N 
 V_{1,n}
 -\left(
 \alpha \epsilon
 -2\epsilon^3 C
 \right)
  \sum_{n=1}^N  \int_0^1 w_{nxx}^2 \, dx 
  \end{equation}
Now choosing $\alpha >  \frac{C  (1+\epsilon N^2\pi^2 + c)^2}{\pi^2}$ we obtain $\dot V \leq -2 c V$. Through norm equivalences, the direct and inverse transformation and the Parseval identity, the result follows.

\end{proof}

In this section we have established a complete framework for finite-dimensional control of the reaction-diffusion equation on a square domain. The framework provides explicit conditions for stabilizability, constructive control laws, and quantitative performance guarantees. Most importantly, it reveals the fundamental tradeoffs between number of actuators, achievable performance, and robustness. These results will prove crucial as we move to more complex geometries where similar modal decompositions arise naturally.

\section{Control of Reaction-Diffusion Equations on Sector Domains}\label{sec:pizza}

Reaction-diffusion processes in non-rectangular domains arise in many physical applications, including heat transfer in pie-shaped regions, chemical diffusion in tapered channels, and neuronal signal propagation in wedge-shaped tissues. In this section, we extend our control methodology to a sector domain, which presents unique challenges due to its radial geometry.

Consider a reaction-diffusion equation on a sector domain, described in polar coordinates $(r,\theta)$ with $r \in [0,R]$ and $\theta \in [\theta_1,\theta_2]$:

\begin{align}
u_t &= \epsilon\left(u_{rr} + \frac{1}{r}u_r + \frac{1}{r^2}u_{\theta\theta}\right) + \lambda u \\
u(t,r,\theta_1) &= u(t,r,\theta_2) = 0 \quad \text{(no flux at edges)} \\
u(t,R,\theta) &= U(t,\theta) \quad \text{(control at outer radius)}
\end{align}

Here, $\epsilon > 0$ represents the diffusion coefficient, $\lambda$ is the reaction coefficient (which may be positive, making the system unstable in open loop), and $U(t,\theta)$ is our control input applied at the outer radius $r = R$. The first equation describes the reaction-diffusion dynamics, with the Laplacian operator expressed in polar coordinates.

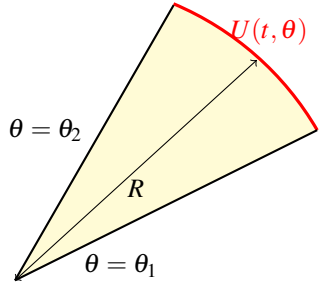
\begin{figure}[!t]
\centering
\begin{tikzpicture}[scale=2]
    \fill[yellow!20] (0,0) -- (2,1) arc(30:67:2) -- cycle;
    
    \draw[red, very thick] (2,1) arc(30:67:2) node[midway, above] {$~~~~U(t,\theta)$};
    
    \draw[thick] (0,0) -- (2,1);
    \draw[thick] (0,0) -- (1.05,1.832);
    
    \draw (0.7,0.1) node {$\theta=\theta_1$};
    \draw (0.2,1) node {$\theta=\theta_2$};
    
    \draw[<->] (0,0) -- (1.6,1.466) node[midway, below] {$R$};
\end{tikzpicture}
\caption{Sector domain with boundary control at the outer radius. The control $U(t,\theta)$ is applied along the curved boundary (shown in red), while homogeneous Dirichlet conditions are imposed on the straight edges.}
\label{fig:sector}
\end{figure}

Figure \ref{fig:sector} illustrates our control setup. The domain resembles a pizza slice with control applied only at the outer curved boundary (the "crust"), depicted in red. No control is available at the straight edges or at the origin.

Our control strategy leverages the natural symmetries of the domain. The first step is to decompose the solution into angular modes, which allows us to transform the 2D problem into a set of 1D problems.

\begin{proposition}[Angular Eigenfunctions]
The functions
\begin{equation}
\Phi_n(\theta) = \sin\left(\frac{n\pi(\theta - \theta_1)}{\theta_2-\theta_1}\right), \quad n=1,2,\ldots
\end{equation}
form a complete orthonormal set in $L^2([\theta_1,\theta_2])$ subject to the boundary conditions $\Phi(\theta_1) = \Phi(\theta_2) = 0$.
\end{proposition}

These eigenfunctions correspond to the natural vibrational modes in the angular direction, analogous to the standing waves on a string with fixed endpoints. The scaling factor $\frac{\pi}{\theta_2-\theta_1}$ adjusts for the angular span of the sector.

Using this basis, we can expand both the solution and control input as:
\begin{align}
u(t,r,\theta) &= \sum_{n=1}^{\infty} u_n(t,r)\sin\left(\frac{n\pi(\theta - \theta_1)}{\theta_2-\theta_1}\right) \\
U(t,\theta) &= \sum_{n=1}^{\infty} U_n(t)\sin\left(\frac{n\pi(\theta - \theta_1)}{\theta_2-\theta_1}\right)
\end{align}

Substituting these expansions into the original PDE and exploiting the orthogonality of the angular eigenfunctions, we obtain a separate radial equation for each mode $n$:

\begin{align}
u_{n,t} &= \epsilon\left(\frac{1}{r}(ru_{n,r})_r - \frac{n^2\pi^2}{(\theta_2-\theta_1)^2}\frac{u_n}{r^2}\right) + \lambda u_n \\
u_n(t,R) &= U_n(t)
\end{align}

Note that each mode experiences an effective reaction term that combines the original reaction coefficient $\lambda$ with a mode-dependent term arising from the angular derivatives. The term $\frac{n^2\pi^2}{(\theta_2-\theta_1)^2}\frac{1}{r^2}$ represents the centrifugal effect that increases with both the mode number $n$ and the proximity to the origin (as $r$ decreases). 

To stabilize the system, we apply the backstepping methodology to each radial mode. For the $n$-th mode, we design a transformation:

\begin{equation}
w_n(t,r) = u_n(t,r) - \int_0^r k_n(r,\rho)u_n(t,\rho)d\rho
\end{equation}

The goal of this transformation is to map our original system into a target system with desired stability properties:

\begin{align}
w_{n,t} &= \epsilon\left(\frac{1}{r}(rw_{n,r})_r - \frac{n^2\pi^2}{(\theta_2-\theta_1)^2}\frac{w_n}{r^2}\right) - cw_n \\
w_n(t,R) &= 0
\end{align}

where $c > 0$ is our desired decay rate. This target system is exponentially stable with decay rate $c$, as the original diffusion operator is enhanced with an additional damping term $-cw_n$.

Following the standard backstepping procedure (substituting the transformation into the original PDE and matching terms with the target system), we determine that the kernel $k_n$ must satisfy the following PDE:

\begin{align}
\frac{\partial^2 k_n}{\partial r^2} + \frac{1}{r}\frac{\partial k_n}{\partial r}
-\frac{\partial^2 k_n}{\partial \rho^2} - \frac{1}{\rho}\frac{\partial k_n}{\partial \rho}
-\frac{k_n}{\rho^2} 
-\alpha_n^2\left(\frac{1}{r^2}-\frac{1}{\rho^2}\right)k_n(r,\rho)
&= \frac{\lambda+c}{\epsilon}k_n(r,\rho) \\
k_n(r,r) &= -\frac{\lambda+c}{2\epsilon r}
\end{align}
where $\alpha_n=\frac{n\pi}{\theta_2-\theta_1}$ is the angular eigenvalue for mode $n$.

Solving the kernel equation directly is challenging due to its variable coefficients. However, by introducing a change of variables motivated by the sector geometry, we can transform it into a more tractable form.

Let $k_n(r,\rho) = g_n(r,\rho)\rho\left(\frac{\rho}{r}\right)^{\alpha_n}$. This transformation captures the geometric scaling inherent in the sector domain. Substituting into the kernel equation yields:

\begin{align}
\partial_{rr}g_n + (1-2\alpha_n)\frac{\partial_r g_n}{r}
-\partial_{\rho\rho}g_n - (1+2\alpha_n)\frac{\partial_\rho g_n}{\rho}
&= \frac{\lambda+c}{\epsilon}g_n \\
g_n(r,r) &= -\frac{\lambda+c}{2\epsilon}
\end{align}

This transformed equation has a remarkable property, captured in the following proposition:

\begin{proposition}[Explicit Kernel Solution]
The function $g_n$ is independent of $n$ and given by:
\begin{equation}
g_n(r,\rho) = -\frac{\lambda+c}{\epsilon}
\frac{I_1\left[\sqrt{\frac{\lambda+c}{\epsilon}(r^2-\rho^2)}\right]}{\sqrt{\frac{\lambda+c}{\epsilon}(r^2-\rho^2)}}
\end{equation}
where $I_1$ is the modified Bessel function of the first kind.

Therefore, the complete kernel is:
\begin{equation}
k_n(r,\rho) = -\frac{\lambda+c}{\epsilon}\rho
\left(\frac{\rho}{r}\right)^{\frac{n\pi}{\theta_2-\theta_1}}
\frac{I_1\left[\sqrt{\frac{\lambda+c}{\epsilon}(r^2-\rho^2)}\right]}{\sqrt{\frac{\lambda+c}{\epsilon}(r^2-\rho^2)}}
\end{equation}
\end{proposition}

This elegant solution reveals how the kernel adapts to both the mode number and the geometry of the sector domain. The modified Bessel function $I_1$ captures the parabolic nature of the diffusion process, while the term $\left(\frac{\rho}{r}\right)^{\frac{n\pi}{\theta_2-\theta_1}}$ accounts for the angular mode's behavior in the radial direction.

Once we have derived the kernel functions for each mode, we can reconstruct the physical space control law:

\begin{align}
U(t,\theta) &= -\sum_{n=1}^N \int_0^R k_n(R,\rho)u_n(t,\rho)\rho d\rho
\sin\left(\frac{n\pi(\theta - \theta_1)}{\theta_2-\theta_1}\right) \\
&= -\int_0^R \int_{\theta_1}^{\theta_2} K(R,\rho,\theta,\eta)u(t,\rho,\eta)\rho d\eta d\rho
\end{align}
where the complete kernel $K$ is:
\begin{align}
K(r,\rho,\theta,\eta) &= 2\sum_{n=1}^N \rho \Bigg(
\frac{\lambda+c}{\epsilon}
\left(\frac{\rho}{r}\right)^{\frac{n\pi}{\theta_2-\theta_1}}
\frac{I_1\left[\sqrt{\frac{\lambda+c}{\epsilon}(r^2-\rho^2)}\right]}{\sqrt{\frac{\lambda+c}{\epsilon}(r^2-\rho^2)}}  \sin\left(\frac{n\pi(\eta - \theta_1)}{\theta_2-\theta_1}\right)
\sin\left(\frac{n\pi(\theta - \theta_1)}{\theta_2-\theta_1}\right)\Bigg)
\end{align}

In practice, we truncate the infinite series to a finite number of modes $N$, which leads to the following stability result:

\begin{theorem}[Exponential Stability of the Closed-Loop System]
Consider the reaction-diffusion system on a sector domain with the feedback control law given by:
\begin{equation}
U(t,\theta) = -\int_0^R \int_{\theta_1}^{\theta_2} K(R,\rho,\theta,\eta)u(t,\rho,\eta)\rho d\eta d\rho
\end{equation}

For the closed-loop system, if the number of controlled modes $N$ satisfies:
\begin{equation}
N > \sqrt{\frac{c+\lambda}{\epsilon}}\frac{(\theta_2-\theta_1)R}{\pi}
\end{equation}
then the zero equilibrium is exponentially stable with decay rate $c$, i.e.:
\begin{equation}
\|u(t,\cdot,\cdot)\|_{L^2(\Omega)} \leq Me^{-ct}\|u(0,\cdot,\cdot)\|_{L^2(\Omega)}
\end{equation}
where $\Omega$ denotes the sector domain and $M > 0$ is a constant.
\end{theorem}

This theorem establishes that by controlling a sufficient number of modes, the closed-loop system achieves exponential convergence to the zero equilibrium with a prescribed decay rate. The required number of modes depends on the domain parameters (angular span and radius), the reaction coefficient, the diffusion coefficient, and the desired decay rate. 

As in the rectangular case, we can consider finite-dimensional actuation in practical implementations, but the detailed analysis is omitted for brevity. The key takeaway is that the control problem on a sector domain can be effectively addressed through modal decomposition and backstepping techniques, resulting in an explicit control law with provable stability properties.

\section{Domain Extension Method for Irregular Domains}\label{sec:domain}

We now address the control of reaction-diffusion equations on domains with irregular geometry. 

Consider as an example the system:

\begin{align}
u_t &= \epsilon(u_{xx} + u_{yy}) + \lambda u, \quad (x,y) \in \Omega \\
u &= 0 \quad \text{on all edges except control boundary} \\
u &= U(t,x) \quad \text{on control boundary}
\end{align}

where $\Omega$ is the piano-shaped domain depicted in Figure \ref{fig:piano}.

\begin{figure}[!t]
\centering
\begin{tikzpicture}[scale=2]
    \coordinate (A) at (0,0);
    \coordinate (B) at (2,0);
    \coordinate (C) at (2,2);
    \coordinate (D) at (0,2);
    \coordinate (E) at (1,2); 
    \coordinate (F) at (0,1); 
    
    \fill[black!10] (A) -- (B) -- (C) -- (E) -- (F) -- cycle;
    
    \draw[thick] (A) -- (B) -- (C) -- (E) -- (F) -- cycle;
    
    \draw[red, very thick] (E) -- (F) node[midway, sloped, above] {$U_1(t,s)$};
    \draw[blue, very thick] (C) -- (E) node[midway, above] {$U_2(t,x)$};
    
    \draw[<->] (0,-0.2) -- (2,-0.2) node[midway, below] {$L$};
    \draw[<->] (2.2,0) -- (2.2,2) node[midway, right] {$L$};
\end{tikzpicture}
\caption{Piano-shaped domain with control at the "back" boundaries}
\label{fig:piano}
\end{figure}
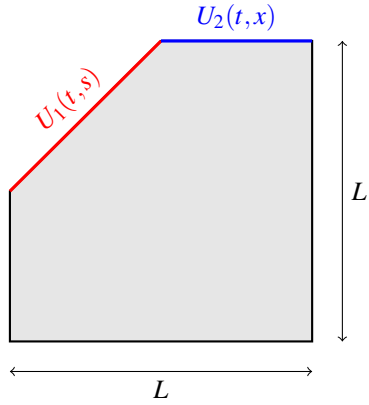

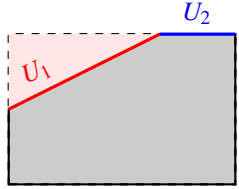
\begin{figure}[t!]
\centering
\begin{tikzpicture}[scale=1]
    \fill[black!20] (0,0) -- (3,0) -- (3,2) -- (2,2) -- (1,1.5) -- (0,1) -- cycle;
    \draw[thick] (0,0) -- (3,0) -- (3,2) -- (2,2) -- (1,1.5) -- (0,1) -- cycle;
    
    \draw[thick, dashed] (0,0) -- (3,0) -- (3,2) -- (0,2) -- cycle;
    
    \fill[red!10] (2,2) -- (1,1.5) -- (0,1) -- (0,2) -- cycle;
    
    \draw[blue, very thick] (3,2) -- (2,2) node[midway, above] {$U_2$};
    \draw[red, very thick] (2,2) -- (1,1.5) -- (0,1) node[pos=0.5, sloped, above] {$U_1$};
\end{tikzpicture}
\caption{General domain extension concept: irregular domain (shaded) extended to regular domain (square) with extension region (colored)}
\label{fig:extension_general}
\end{figure}

The key insight for controlling such a domain is to extend it to a simpler domain where known control techniques can be applied. This leads to our main methodological contribution, the \emph{Domain Extension} technique for solving boundary control problems. Thus, the original control problem on $\Omega$ can be transformed into an equivalent problem on the extended square domain $[0,L]\times[0,L]$ through appropriate choice of controls $U_1$ and $U_2$ on the irregular boundary.

The domain extension approach proceeds in several steps. Consider the extension of $\Omega$ to the square $[0,L]\times[0,L]$. Let $\Omega_e$ denote the extension region (the "cut-off" triangle, see Fig.~\ref{fig:extension_general}). In addition:
\begin{enumerate}
\item The dynamics in $\Omega_e$ are simulated using the same PDE.
\item The boundary conditions for $\Omega_e$ are: the value of the normal derivative of $u$ (flux) at the interface, which can be obtained from the original equation, zero in the side and the control law computed from~(\ref{control-law-square}) where the value of the state comes from the extended domain.
\item Then solution in $\Omega_e$ is used to compute the value for $U_1$.
\end{enumerate}

Since the solutions match at the interface, then by uniqueness, the solution in $\Omega$ solves the original problem and the following theorem can be stated.
\begin{theorem}
If the extended system is exponentially stable with decay rate $c$, then the solution in the original piano domain satisfies
\begin{equation}
\|u(t,\cdot,\cdot)\|_{L^2(\Omega)} \leq Me^{-ct}\|u(0,\cdot,\cdot)\|_{L^2(\Omega)}
\end{equation}
where $M$ depends only on the geometry of $\Omega$.
\end{theorem}

\begin{proof}
Let $v(t,x,y)$ be the solution in the extended square domain. By construction, $u(t,x,y) = v(t,x,y)$ for $(x,y) \in \Omega$. Since $v$ satisfies
\begin{equation}
\|v(t,\cdot,\cdot)\|_{L^2([0,L]^2)} \leq Me^{-ct}\|v(0,\cdot,\cdot)\|_{L^2([0,L]^2)}
\end{equation}
and $\Omega$ is a subset of $[0,L]^2$, the result follows immediately.
\end{proof}

The practical implementation requires solving two coupled problems, namely, the simulation of the extension domain dynamics (with complementary boundary conditions at the interface), and the computation of the backstepping control law for the full extended domain based on real and virtual values.

The domain extension method has several important properties:
\begin{enumerate}
\item Preserves the parabolic smoothing properties of the original PDE
\item Requires no explicit boundary conditions at the interface
\item Maintains stability margins of the original backstepping design
\end{enumerate}

In general, if we consider the case where the entire boundary of the original domain is controlled, the mechanism of control works as follows for any possible domain:
\begin{enumerate}
\item We choose an extended domain where we know how to solve the control problem (this could be referred to as a ``target domain'' as an analogy to the backstepping target system).
\item We design a control law for the extended regular domain $\Omega_R$
\item We implement this control on the outer boundary of $\Omega_R$
\item We establish a ``transfer'' mechanism at the interface $\Gamma_i$ between the real and virtual domains
\end{enumerate}

It must be mentioned that calling the extension domain the ``target domain'' is not merely an allusion to the target equation concept. Indeed the behaviour of the system (the solution) would be ``as if'' the domain was indeed the chosen one for the extension.

This transfer mechanism depends on the type of boundary conditions: For Dirichlet boundary conditions: The state values $u(t,x)$ at the interface $\Gamma_i$ are used as inputs to the virtual domain and the derivatives (normal flow) $\frac{\partial u}{\partial n}(t,x)$ at the interface, computed from the Neumann map in the virtual domain, determine the control law on $\Gamma_c$. For Neumann boundary conditions:  The derivatives $\frac{\partial u}{\partial n}(t,x)$ at the interface are used as inputs to the virtual domain and the state values $u(t,x)$ at the interface, computed from the Dirichlet map in the virtual domain, determine the control law on $\Gamma_c$

Consider, for instance, an irregular domain $\Omega$ that we wish to extend to a disk:
\begin{equation}
\Omega = \{(r,\theta): 0 \leq r \leq R(\theta), \theta \in [0,2\pi]\}
\end{equation}
where $R(\theta)$ is a piecewise smooth function defining the irregular boundary.

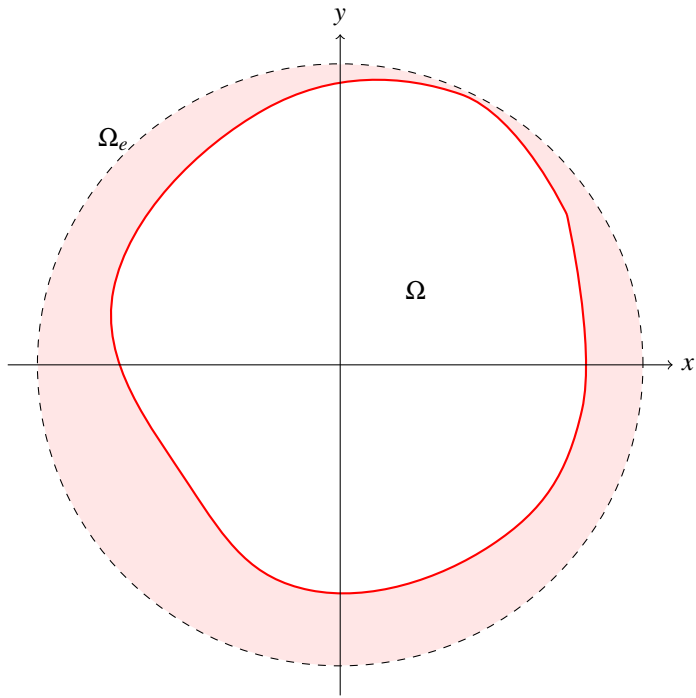
\begin{figure}[!t]
\centering
\begin{tikzpicture}[scale=2]
    \draw[thick,dashed] (0,0) circle (2);
    
    \draw[thick] plot[smooth, tension=0.7] 
        coordinates {(1.5,1) (0.8,1.8) (-0.5,1.7) (-1.5,0.5) 
        (-1,-0.8) (-0.2,-1.5) (1,-1.2) (1.6,-0.3) (1.5,1)};
    \fill[black!10] plot[smooth, tension=0.7] 
        coordinates {(1.5,1) (0.8,1.8) (-0.5,1.7) (-1.5,0.5) 
        (-1,-0.8) (-0.2,-1.5) (1,-1.2) (1.6,-0.3) (1.5,1)};
    
    \fill[red!10] (0,0) circle (2);
    \fill[white] plot[smooth, tension=0.7] 
        coordinates {(1.5,1) (0.8,1.8) (-0.5,1.7) (-1.5,0.5) 
        (-1,-0.8) (-0.2,-1.5) (1,-1.2) (1.6,-0.3) (1.5,1)};
    
    \draw[red,thick] plot[smooth, tension=0.7] 
        coordinates {(1.5,1) (0.8,1.8) (-0.5,1.7) (-1.5,0.5) 
        (-1,-0.8) (-0.2,-1.5) (1,-1.2) (1.6,-0.3) (1.5,1)};
    
    \draw[->] (-2.2,0) -- (2.2,0) node[right] {$x$};
    \draw[->] (0,-2.2) -- (0,2.2) node[above] {$y$};
    
    \node at (0.5,0.5) {$\Omega$};
    \node at (-1.5,1.5) {$\Omega_e$};
\end{tikzpicture}
\caption{Extension of an irregular domain to the target domain of a disk. Original domain $\Omega$ (shaded), extension region $\Omega_e$ (patterned), and control boundary (red).}
\label{fig:disk_extension}
\end{figure}

In this case, the target or extension domain $\Omega_R$ is the disk of radius $R_{\max} = \max_{\theta} R(\theta)$ as shown in Figure~\ref{fig:disk_extension}. The control problem becomes particularly elegant as we can exploit the radial symmetry of the extended domain and apply e.g. the designs of~\cite{nball} while accommodating the irregular geometry of the original domain.

This approach can be generalized to higher dimensions. For instance, in $\mathbb{R}^3$, we can extend a truncated ball to a complete ball and utilize spherical harmonics for the solution representation. The beauty of this approach lies in its versatility: virtually any domain with a Lipschitz boundary can be extended to a regular domain where our control techniques apply. A complete mathematical treatment, including explicit construction of the extension operators and analysis of the interface conditions, will be presented in future work.

\section{Conclusions}\label{sec:concl}
This paper has presented a systematic framework for controlling reaction-diffusion PDEs in higher dimensions through the exploitation of domain symmetries and geometric properties. Three fundamental approaches have been developed:

The first approach, based on Fourier analysis in rectangular domains, provides explicit conditions for stabilization with finite-dimensional actuation. The relationship between number of actuators and achievable performance was precisely characterized.

The second approach addresses sector domains through angular eigenfunction expansions, yielding explicit kernel solutions in terms of modified Bessel functions. The geometric properties of the sector lead to enhanced stability properties for higher modes.

Finally, the domain extension methodology provides an idea for irregular domains, transforming the original problem into one (a target domain) where known control techniques can be applied. This approach maintains the stability properties of the original design while accommodating complex geometries.

These results establish a comprehensive framework for PDE control in higher dimensions, providing both theoretical understanding and practical implementation guidelines.

Recent work has also explored the use of machine learning techniques to approximate backstepping kernels \cite{krstic2023neural}, opening new avenues for practical implementation.


\section*{Funding}
R. Vazquez acknowledges support of grant PID2023-147623OB-I00 funded by MICIU/AEI/10.13039/ 501100011033 and by ``ERDF A way of making Europe.''

\end{document}